\documentclass[letterpaper, 10 pt, conference]{ieeeconf}  % Comment this line out if you need a4paper

\IEEEoverridecommandlockouts                              % This command is only needed if
\usepackage{graphicx}
\usepackage{amsmath}
\usepackage{amssymb}
\usepackage{comment}
\usepackage{color}
\usepackage{epstopdf}
\usepackage{float}
\usepackage{mathrsfs}
\usepackage{hyperref}

\makeatletter
\let\theoremstyle\@undefined                        % undefine \theoremstyle
\makeatother

\usepackage{amsthm}

\newtheorem{nntheorem}{\bf Theorem}
\newtheorem{nnassumption}{\bf Assumption}
\newtheorem{nndefinition}{\bf Definition}
\newtheorem{nnlemma}{\bf Lemma}
\newtheorem{nncorollary}{\bf Corollary}
\newtheorem{nnproposition}{\bf Proposition}
\newtheorem{nexample}{\bf Example}

\newenvironment{nnexample}
{\begin{nexample}\rm}{\end{nexample}}

\newtheorem{nnremark}{\bf Remark}

\def\NN{\mathbb{N}}

\def\epsilon{\varepsilon}

\title{\LARGE \bf
Stabilization of a linear Korteweg-de Vries equation\\
with a saturated internal control
}

\author{Swann Marx$^{1}$, Eduardo Cerpa$^{2}$, Christophe Prieur$^{3}$ and Vincent Andrieu$^{4}$% <-this % stops a space
\thanks{$^{1}$Swann Marx is with Gipsa-lab, Department of Automatic Control, Grenoble Campus, 11 rue des Math\'ematiques, BP 46, 38402 Saint Martin d'H\`eres Cedex,
        {\tt\small marx.swann@gmail.com} This work has been partially supported by the LabEx PERSYVAL-Lab (ANR-11-LABX-0025-01), by  the ANR project LimICoS contract number 12-BS03-005-01, by MathAmsud COSIP, by CONICYT grant ACT 1106
Basal Project FB0008 and by AC3E, UTFSM.}%
\thanks{$^{2}$Eduardo Cerpa is with  Departamento de Matem\'atica, Universidad T\'ecnica Federico
Santa Mar\'ia, Avda. Espa\~na 1680, Valpara\'iso, Chile 
{\tt\small eduardo.cerpa@usm.cl}.}%
\thanks{$^{3}$Christophe Prieur is with Gipsa-lab, Department of Automatic Control, Grenoble Campus, 11 rue des Math\'ematiques, BP 46, 38402 Saint Martin d'H\`eres Cedex,
   {\tt\small christophe.prieur@gipsa-lab.fr}.}
   \thanks{$^{4}$Vincent Andrieu is with  Universit\'e Lyon 1 CNRS UMR 5007 LAGEP, France and Fachbereich C - Mathematik und Naturwissenschaften, Bergische Universit\"at Wuppertal, Gau\ss stra\ss e 20, 42097 Wuppertal, Germany {\tt\small vincent.andrieu@gmail.com}.}
}

\begin{document}

\maketitle
%\thispagestyle{empty}
%\pagestyle{empty}

%%%%%%%%%%%%%%%%%%%%%%%%%%%%%%%%%%%%%%%%%%%%%%%%%%%%%%%%%%%%%%%%%%%%%%%%%%%%%%%%
\begin{abstract}
This article deals with the design of saturated controls in the context of partial differential equations. It is focused on a linear Korteweg-de Vries equation, which is a mathematical model of waves on shallow water surfaces. In this article, we close the loop with a saturating input that renders the equation nonlinear. The well-posedness is proven thanks to the nonlinear semigroup theory. The proof of the asymptotic stability of the closed-loop system  uses a Lyapunov function. %together with a sector condition describing the saturating input.
\end{abstract}

%%%%%%%%%%%%%%%%%%%%%%%%%%%%%%%%%%%%%%%%%%%%%%%%%%%%%%%%%%%%%%%%%%%%%%%%%%%%%%%


\section{Introduction}\label{introduction}

%to be cited: \cite{tucsnak2009observation}, \cite{brezis1973operateurs}, \cite{luo1999stabinf}.

In recent decades, a great effort has been made to take into account input saturations in control designs (see e.g \cite{tarbouriech2011book_saturating} or \cite{zacc2003antiwindupLMI}). Indeed, in most of systems, actuators are limited due to some physical constraints and the control input has to be bounded. Neglecting the amplitude actuator limitation can be source of undesirable and catastrophic behaviors for the closed-loop system. The standard method follows a two steps design. First the design is carried out without taking into account the saturation. In a second step, a nonlinear analysis of the closed loop system is made when adding the saturation. In this way, we often get local stabilization results. Tackling this particular nonlinearity in the case of finite dimensional systems is already a difficult problem. However, nowadays, numerous techniques are now available (see e.g. \cite{tarbouriech2011book_saturating,teel1992globalsaturation,sussmann1991saturation}) and such systems 
 can be analyzed with an appropriate Lyapunov function and a sector condition of the saturation map, as introduced in \cite{tarbouriech2011book_saturating}. 

To the best of our knowledge, there are few papers studying this topic in the infinite dimensional case. Among them, we find \cite{lasiecka2002saturation} and more recently \cite{prieur2014wave_saturating}, where a wave equation equipped with a saturated distributed actuator is considered. Note that saturation function can be defined with a sign function, which is also used in sliding mode control design theory. The interest reader can refer to \cite{pisano2011tracking,cristofaro2014robust}, where a wave and a reaction-diffusion equations are stabilized with a sliding mode controller. The present paper aims at contributing to the study of the saturated input case in the framework of partial differential equations.
 
Let us note that in \cite{slemrod1989mcss} the case of a priori bounded feedback is studied for abstract linear systems. To be more specific, for compact control operators, some conditions are derived to deduce, from the asymptotic stability of an infinite-dimensional linear system in abstract form, the asymptotic stability when closing the loop with saturating controller (see \cite[Theorem 5.1]{slemrod1989mcss} for a precise statement of this result). The aim of our article is to study a particular partial differential equation without seeing it as an abstract control system and without checking the very specific assumptions of \cite{slemrod1989mcss}.
 
The Korteweg-de Vries equation (KdV for short)
\begin{equation}
y_t+y_{x}+y_{xxx}+yy_x=0,
\end{equation}
is a mathematical model of waves on shallow water surfaces. Its controllability and stabilizability properties have been deeply studied in the case with no constraints on the control, as explained in 
 \cite{cerpa2013control, bible_coron, rosier-zhang}. In this article, we focus on the following controlled linear KdV equation
\begin{equation}
\label{kdv-semiperiodic}
\left\{
\begin{split}
&y_t+y_x+y_{xxx}+f=0,\: (x,t)\in [0,L]\times[0,+\infty),\\
&y(t,0)=y(t,L)=y_x(t,L)=0,\: t\in[0,+\infty),\\
%&y_x(t,L)=0,\: t\in[0,+\infty),\\
&y(0,x)=y_0(x),\: x\in [0,L],
\end{split}
\right.
\end{equation}
where $y$ stands for the state and $f$ for the control. As studied in \cite{rosier1997kdv}, if $f=0$ and
\begin{equation*}
L\in \left\{ 2\pi\sqrt{\frac{k^2+kl+l^2}{3}}\,\Big\slash \,k,l\in\mathbb{N}^*\right\},
\end{equation*}
then, there exist solutions of \eqref{kdv-semiperiodic} for which the energy does not decay to zero. For instance, if $L=2\pi$ and $y_0=1-\cos(x)$ for all $x\in [0,L]$, then $y(t,x)=1-\cos(x)$ is a stationary solution of \eqref{kdv-semiperiodic} conserving the energy for any time $t$. In the literature there are some methods stabilizing the KdV equation \eqref{kdv-semiperiodic} with boundary  \cite{cerpa2009rapid, cerpa_coron_backstepping, marx-cerpa} or internal controls \cite{perla-vasconcellos-zuazua, pazoto2005localizeddamping}. Here we focus on the internal control case. In fact, as proven in \cite{perla-vasconcellos-zuazua, pazoto2005localizeddamping}, the feedback control $f(t,x)=a(x)y(t,x)$, where $a=a(x)$ is a positive function whose support is a nonempty open subset of $(0,L)$, makes the origin an exponentially stable state. 

The question we want to address is the following. Given a feedback control $f$ stabilizing the equation, what do we get if we saturate it? Is the equation still stable? We deal with the case in which $f(t,x)=ay(t,x)$ with a positive constant $a$ and we show that the origin is asymptotically stable for the closed-loop system with a saturated input. 

This article is organized as follows. In Section \ref{main}, we present our main results about the well posedness and the stability of this equation in presence of saturation. Section \ref{proof} is devoted to prove these results by using the nonlinear semigroup theory and Lyapunov techniques. In Section \ref{simulation}, we give some simulations of the equation looped by a saturated feedback. Section \ref{part5} collects some concluding remarks and possible further research lines. 

\textbf{Notation:} $y_t$ (resp. $y_x$) stands for the partial derivative of the function $y$ with respect to $t$ (resp. $x$) (this is a shortcut for $\frac{\partial y}{\partial t}$, resp. $\frac{\partial y}{\partial x}$). $\mathfrak{R}$ (resp. $\mathfrak{I}$) denotes the real (resp. imaginary) part of a complex number. Given $L>0$, $\Vert \cdot\Vert_{L^2(0,L)}$ denotes the norm in $L^2(0,L)$ and  $H^3(0,L)$ is the set of all functions $u\in L^2(0,L)$ such that $u_x, u_{xx},u_{xxx}\in L^2(0,L)$. Finally $H_0^1(0,L)$ is the closure in $L^2(0,L)$ of the set of smooth functions that are vanishing at $x=0$ and $x=L$. It is equipped with the norm $\Vert u\Vert^2_{H_0^1(0,L)}:=\int_0^L |u_x|^2dx$. The associate inner products are denoted $\langle \cdot,\cdot\rangle_{L^2(0,L)}$ and $\langle \cdot,\cdot\rangle_{H_0^1(0,L)}$. $H^3_L(0,L)$ denotes the set $H^3_L(0,L):=\lbrace w\in H^3(0,L),\: w(0)=w(L)=w^\prime(L)=0\rbrace$.

\section{Main results}\label{main}
For any $a>0$, if we take $f(t,x):=ay(t,x)$ in \eqref{kdv-semiperiodic}, then we get that the equation is stabilized. Indeed, any solution of
\begin{equation}
\label{linear-KdV-controlled}
\left\{
\begin{split}
&y_t+y_{xxx}+y_x+ay=0,\\
&y(t,0)=y(t,L)=0,\\
&y_x(t,L)=0,
\end{split}
\right.
\end{equation}
satisfies
\begin{equation}
\begin{split}
\frac{1}{2}\frac{d}{dt}\int_0^L |y(t,x)|^2 dx=&-\frac{1}{2}|y_x(t,0)|^2-a\int_0^L |y(t,x)|^2 dx\\
\leq & -a\int_0^L |y(t,x)|^2 dx,
\end{split}
\end{equation} 
which ensures an exponential stability with respect to the $L^2(0,L)$-norm. Note that the decay rate can be selected as large as we want by tuning the parameter $a$. Such a result is refered to as a rapid stabilization result.  

Let us assume now that the control is constrained and that we have to consider the following feedback law
\begin{equation}
f(t,x)=a\cdot \texttt{sat}(y(t,x))
\end{equation}
where the function $\texttt{sat}$ is defined by

\begin{equation}
\label{function-saturation}
\texttt{sat}(s)=\left\{
\begin{array}{rl}
-u_{\min}&\text{if }s< -u_{\min},\\
s&\text{if }-u_{\min}\leq s\leq u_{\max},\\
u_{\max}&\text{if } s> u_{\max}.
\end{array}
\right.
\end{equation}

To ease the lecture, we assume same levels of saturation, which means that $u_{\max}=u_{\min}=u_0$. 

%Introducing the function $\varphi$ called dead-zone, defined by
%\begin{equation}
%\varphi(s):=\texttt{sat}(s)-s,
%\end{equation}
We can write the KdV equation controlled by a saturated control as follows
\begin{equation}
\label{KdV-saturated}
\left\{
\begin{split}
&y_t+y_{xxx}+y_x+a\texttt{sat}(y)=0,\\
&y(t,0)=y(t,L)=0,\\
&y_x(t,L)=0,\\
&y(0,x)=y_0(x).
\end{split}
\right.
\end{equation}

Let us state the main results of this paper.
\begin{nntheorem}[Well-posedness]
\label{theorem-wp}
For any initial condition $y_0\in H_L^3(0,L)$, there exists a unique strong continuous solution $y:[0,\infty)\rightarrow H_L^3(0,L)$ to (\ref{KdV-saturated}) that is continuous from $[0,\infty)$ to $H_L^3(0,L)$ and continuously differentiable from $[0,\infty)$ to $L^2(0,L)$.

Moreover, for any initial condition $y_0$ in $L^2(0,L)$, there exists a unique weak solution $y:[0,\infty)\rightarrow L^2(0,L)$ to (\ref{KdV-saturated}) that is continuous from $[0,\infty)$ to $L^2(0,L)$.
\end{nntheorem}

\begin{nntheorem}[Asymptotic stability]
\label{theorem-s}
For any constant $a>0$, the equation \eqref{KdV-saturated} is globally asymptotically stable. More precisely, the following property holds. For any initial condition $y_0$ in $L^2(0,L)$, the weak solution to (\ref{KdV-saturated}) satisfies 
\begin{equation}
\Vert y(t,\cdot)\Vert_{L^2(0,L)}\leq \Vert y_0\Vert_{L^2(0,L)},\, \forall t\geq 0,
\end{equation}
together with the attractivity property
\begin{equation}
\label{convergence-property}
\Vert y(t,\cdot)\Vert_{L^2(0,L)}\rightarrow 0,\:\text{ as }t\rightarrow \infty.
\end{equation}
\end{nntheorem}

\begin{nnremark}
The exponential stability of the closed-loop system with a saturating control is an open problem for the KdV equation.\end{nnremark}

\section{Proof of Theorems \ref{theorem-wp} and \ref{theorem-s}}\label{proof}

\subsection{Well-posedness (Theorem \ref{theorem-wp})}

Let $A$ denote the operator $$Aw=(-w'-w'''-a\texttt{sat}(w))$$ on the domain 
$D(A)\subset L^2(0,L)$ defined such that $D(A):=H^3_L(0,L)$. 

%$$ D(A):=\left\{w\in
%%H^3(0,L) \,\slash\, w(0)=w(L)=w^\prime(L)=0 \right\}.$$

\begin{nnlemma}
\label{lemma-closedness}
Operator $A$ is  closed.
\end{nnlemma}

\begin{proof}
Let $\{u_n\}_{n\in\NN}$ be a sequence in $D(A)$ such that
\begin{equation}
\lim_{n\rightarrow +\infty} u_n=u
\end{equation}
 and
\begin{equation}
\lim_{n\rightarrow +\infty} Au_n=v
\end{equation}
for some $u,v\in L^2(0,L)$. To prove that $A$ is closed, we have to prove that $u\in D(A)$ and that $Au=v$. Let us note that
\begin{equation}
\label{oplin}
\tilde{A}: w\in D(A)\subset L^2(0,L)\mapsto (-w^\prime-w^{\prime\prime\prime})\in L^2(0,L)
\end{equation}
is already closed in $D(A)$. Moreover, we know that the function $\texttt{sat}$ is  globally Lipschitz\footnote{Indeed, we know from \cite[Page 73]{bible_khalil} that for all $(y,\tilde{y})\in L^2(0,L)^2$ and for all $x\in [0,L]$, $|\texttt{sat}(y(x))-\texttt{sat}(\tilde{y}(x))|\leq |y(x)-\tilde{y}(x)|$. Thus we get $\Vert \texttt{sat}(y)-\texttt{sat}(\tilde{y})\Vert_{L^2(0,L)}\leq \Vert y-\tilde{y}\Vert_{L^2(0,L)}$.}. Thus\footnote{Given $T_1$ closed and $T_2$ globally Lipschitz, and $u_n\rightarrow u$ and $(T_1+T_2)u_n\rightarrow w$, we have $|T_1u_n+T_2u-w|\leq |T_1u_n+T_2u_n-w|+|T_2u_n-T_2u|$. Thus, the left member of the inequality is bounded by a term which converges to $0$.} $A$ is closed.
\end{proof}

\begin{nnlemma}
\label{lemma-dissipativity}
Operator $A$ is dissipative.
\end{nnlemma}

\begin{proof}
Let us consider
\begin{equation}
\texttt{sat}_{\mathbb{C}}(s):=\texttt{sat}(\mathfrak{R}(s))+i\texttt{sat}(\mathfrak{I}(s))
\end{equation}
which we will denote by $\texttt{sat}(s)$ to ease the notation.

Given $u,\tilde{u}\in D(A)$, we have that $$\psi(u,\tilde{u}):=\langle Au-A\tilde{u},u-\tilde{u} \rangle_{L^2(0,L)}$$ is equal to 
\begin{equation}
\begin{split}
\psi(u,\tilde{u})= &-\int_0^L \left(u^{\prime\prime\prime}(x)+u^\prime (x)+a\texttt{sat}(u(x))\right.\\
&\left.-(\tilde{u}^{\prime\prime\prime}(x)+\tilde{u}^\prime (x)+a\texttt{sat}(u(x))\right).(\overline{(u-\tilde{u})}(x)) dx\\
= &-a\int_0^L (\texttt{sat}(u)-\texttt{sat}(\tilde{u}))(\overline{u-\tilde{u}})dx\\
&-\int_0^L (u^{\prime\prime\prime}-\tilde{u}^{\prime\prime\prime})(\overline{u-\tilde{u}})dx
\end{split}
\end{equation}
Integrating by parts $\int_0^L (u^{\prime\prime\prime}-\tilde{u}^{\prime\prime\prime})(\overline{u-\tilde{u}})$, we get
\begin{equation}
\begin{split}
\mathfrak{R}\left\{ \int_0^L (u^{\prime\prime\prime}-\tilde{u}^{\prime\prime\prime})(\overline{u-\tilde{u}})\right\}= -|u'(0)|^2
\leq  0.
\end{split}
\end{equation} 
Then we have
\begin{equation}
\begin{split}
&\mathfrak{R}\left\{\langle Au-A\tilde{u},u-\tilde{u} \rangle_{L^2(0,L)}\right\}\leq\\
&-a\mathfrak{R}\left\{ \int_0^L (\texttt{sat}(u)-\texttt{sat}(\tilde{u}))(\overline{u-\tilde{u}})dx\right\}.
\end{split}
\end{equation}
By definition of the saturation function, we get that for all $(s,\tilde{s})\in\mathbb{C}^2$
\begin{equation}
\mathfrak{R}\left\{(\texttt{sat}(s)-\texttt{sat}(\tilde{s}))(\overline{s-\tilde{s}})\right\}\geq 0.
\end{equation}
Thus, thanks to the positivity of $a$, we get that
\begin{equation}
\mathfrak{R}\left\{\langle Au-A\tilde{u},u-\tilde{u} \rangle_{L^2(0,L)}\right\}\leq 0,
\end{equation}
which means that the operator $A$ is dissipative. It concludes the proof of Lemma \ref{lemma-dissipativity}.\end{proof}

In order to conclude the proof of the well-posedness, we have to verify whether the operator $A$ generates a semigroup of contractions which will be denoted in the following by $S(t)$. Following \cite{miyadera1992nl_sg}, we see that it is enough to prove that for all $\lambda>0$ sufficiently small
\begin{equation}
D(A)\subset\texttt{Ran}(I-\lambda A),
\end{equation}
where $\texttt{Ran}$ stands for the range and $I$ for the identity operator. In other words, for each $u\in D(A)$, there exists $\tilde{u}\in D(A)$ such that
\begin{equation}
(I-\lambda A)\tilde{u}=u,
\end{equation}
which is equivalent to prove the existence of a solution of a nonhomogeneous nonlinear equation in the $\tilde{u}$-variable with boundary conditions as considered in the following lemma. 
\begin{nnlemma} 
\label{lemma-existence}
Let us introduce $\tilde{\lambda}:=\frac{1}{\lambda}$. If $a$ is strictly positive and $\tilde{\lambda}>0$, then there exists $\tilde{u}$ solution of
\begin{equation}
\left\{
\begin{split}
&\tilde{\lambda}\tilde{u}+\tilde{u}^{\prime\prime\prime}+\tilde{u}^\prime+ a\texttt{sat}(\tilde{u})=\tilde{\lambda}u,\\
&\tilde{u}(0)=\tilde{u}(L)=\tilde{u}^\prime(L)=0.
\end{split}
\right.
\end{equation}

\end{nnlemma}

\begin{proof}

The proof of this lemma follows from classical technics (see e.g. \cite[Page 179]{miyadera1992nl_sg}) and uses the Schauder fixed-point theorem (see e.g. \cite[Theorem B.19,]{bible_coron}). 

First, following \cite{coroncrepeau2004missed}, let us focus on the spectrum of $\tilde{A}$ which is defined by (\ref{oplin}). Since the operator $\tilde{A}$ has a compact resolvent, its spectrum denoted by $\sigma(\tilde{A})$ consists only of eigenvalues. Futhermore, the spectrum is a discrete subset of $i\mathbb{R}$.  

Since $\tilde{\lambda}$ belongs to $\mathbb{R}_+$, then $\tilde{\lambda}\notin\sigma(\tilde{A})$. Hence $(\tilde{A}-I\tilde{\lambda})$ is invertible and there exists a unique function $z=z(x)$ solution of
\begin{equation}
\label{fixed-point-solve}
\left\{
\begin{split}
&\tilde{\lambda}z+ z^{\prime\prime\prime}+ z^\prime=g,\\
&z(0)=z(L)=z^\prime(L)=0,
\end{split}
\right.
\end{equation}
where $g(y):=-a\texttt{sat}(y)+\tilde{\lambda}u$.

Then we can focus on the map
\begin{equation}
\begin{split}
\mathcal{T}:\: &L^2(0,L)\rightarrow L^2(0,L)\\
& y\longmapsto z=\mathcal{T}(y)
\end{split}
\end{equation}
where $z=\mathcal{T}(y)$ is the unique solution to (\ref{fixed-point-solve}). We define
\begin{equation}
C=\lbrace u\in H_0^1(0,L)\slash\, \Vert u\Vert_{H_0^1(0,L)}\leq M\rbrace
\end{equation}
where $M>0$. From the theorem of Rellich (see \cite[Theorem 9.16, p. 285]{brezis2010functional}), the injection of $H_0^1(0,L)$ in $L^2(0,L)$ is compact, then $C$ is bounded in $H_0^1(0,L)$ and is relatively compact in $L^2(0,L)$. Moreover, it is a closed subset of $L^2(0,L)$. Thus $C$ is a compact subset of $L^2(0,L)$. In order to apply the Schauder theorem, we have to prove that $\mathcal{T}(L^2(0,L))\subset C$ for a suitable choice of $M>0$.  We multiply the first line of (\ref{fixed-point-solve}) by $z$ and then integrate between $0$ and $L$. After some integrations by parts, we get
\begin{equation}
\begin{split}
\tilde{\lambda}\Vert z\Vert_{L^2(0,L)}= &-\int_0^L zz^{\prime\prime\prime}dx-\int_0^L zz^\prime dx\\
&-a\int_0^L \texttt{sat}(y)z dx+\tilde{\lambda}\int_0^L uz dx\\
= &-z^\prime(0)^2-a\int_0^L \texttt{sat}(y)zdx+\tilde{\lambda}\int_0^L uzdx\\
\leq &-a\int_0^L \texttt{sat}(y)zdx+\tilde{\lambda}\int_0^L uzd
\end{split}
\end{equation}
The Young inequality leads us to the following inequality
\begin{equation}
\begin{split}
\tilde{\lambda}\Vert z\Vert^2_{L^2(0,L)}\leq & a\varepsilon_1\Vert\texttt{sat}(y)\Vert^2_{L^2(0,L)}+\frac{a}{\varepsilon_1}\Vert z\Vert_{L^2(0,L)}\\
&+\frac{\tilde{\lambda}}{\varepsilon_2}\Vert z\Vert^2_{L^2(0,L)}+\tilde{\lambda}\varepsilon_2\Vert u\Vert^2_{L^2(0,L)}
\end{split}
\end{equation}
where $\varepsilon_1,\varepsilon_2>0$ are to be chosen later.

The function $\texttt{sat}(\cdot)$ being bounded, we get
\begin{equation}
\left(\tilde{\lambda}-\frac{a}{\epsilon_1}-\frac{\tilde{\lambda}}{\varepsilon_2}\right)\Vert z\Vert^2_{L^2(0,L)}\leq a\varepsilon_1Lu_0^2+\tilde{\lambda}\varepsilon_2\Vert u\Vert^2_{L^2(0,L)}
\end{equation}
We choose $\varepsilon_1$ and $\varepsilon_2$ such that $\alpha:=\left(\tilde{\lambda}-\frac{a}{\epsilon_1}-\frac{\tilde{\lambda}}{\varepsilon_2}\right)>0$. Thus we obtain
\begin{equation}
\label{l3}
\Vert z\Vert^2_{L^2(0,L)}\leq \frac{a\varepsilon_1Lu_0^2}{\alpha}+\frac{\tilde{\lambda}\varepsilon_2}{\alpha}\Vert u\Vert^2_{L^2(0,L)}
\end{equation}
and therefore the $L^2$-norm of $z$ is bounded by a constant.

Now, let us multiply the first line of (\ref{fixed-point-solve}) by $xz$ and then integrate between $0$ and $L$ to get
\begin{equation}
\label{fixed-point-final}
\int_0^L xzz^{\prime\prime\prime}dx+\int_0^L xzz^\prime dx+\tilde{\lambda}\int_0^L xz^2dx=\int_0^L xzg dx
\end{equation}

After some integrations by parts, we get
\begin{equation}
\begin{split}
\label{lastuce}
\int_0^L xzz^{\prime\prime\prime}dx= &-\int_0^L zz^{\prime\prime}dx-\int_0^L xz^{\prime}z^{\prime\prime}dx\\
= &\frac{3}{2}\Vert z^\prime\Vert^2_{L^2(0,L)}
\end{split}
\end{equation}
and
\begin{equation}
\label{lastuce2}
\int_0^L xzz^\prime dx=-\frac{1}{2}\Vert z\Vert^2_{L^2(0,L)}.
\end{equation}
Thus, plugging (\ref{lastuce}) and (\ref{lastuce2}) in (\ref{fixed-point-final}), we obtain
\begin{equation}
\begin{split}
\frac{3}{2}\Vert z^\prime\Vert^2_{L^2(0,L)}= & \frac{1}{2}\Vert z\Vert^2_{L^2(0,L)}-\tilde{\lambda}\int_0^L xz^2dx+\int_0^L xzg\\
\leq & \frac{1}{2}\Vert z\Vert^2_{L^2(0,L)}+\frac{1}{2}\Vert z\Vert^2_{L^2(0,L)}+\frac{L^2}{2}\Vert g\Vert^2_{L^2(0,L)}\\
\leq &\Vert z\Vert^2_{L^2(0,L)}+\frac{L^2}{2}\Vert g\Vert^2_{L^2(0,L)}\\
\leq & M
\end{split}
\end{equation}
where $M$ is a constant which depends only on $u_0$, $L$ and $a$. In this way we see that  $\Vert z^\prime\Vert^2_{L^2(0,L)}$ is also bounded. 

From the Poincar\'e inequality, we have the equivalence between $\Vert z^\prime\Vert^2_{L^2(0,L)}$ and $\Vert z\Vert^2_{H_0^1(0,L)}$. Thus, for any $\lambda > 0$, $a>0$ and $y\in L^2(0,L)$, there exists $M>0$ such that we have $z\in C$, i.e. $\mathcal{T}(L^2(0,L))\subset C$. Then we apply the Schauder theorem (see e.g. \cite[Theorem B.19]{bible_coron}). Hence it concludes the proof of Lemma \ref{lemma-existence}.
\end{proof}

Thus, from \cite[Theorem 4.2]{miyadera1992nl_sg}, using the Lemma \ref{lemma-closedness}, \ref{lemma-dissipativity} and \ref{lemma-existence}, $A$ generates a semigroup of contraction $T(t)$. With \cite[Theorem 3.1]{brezis1973operateurs}, the proof of Theorem \ref{theorem-wp} is achieved.

\subsection{Asymptotic stability (Theorem \ref{theorem-s})}
The Lyapunov function related to (\ref{KdV-saturated}), which we will denote by $E$, is given by
\begin{equation}
E:=\frac{1}{2}\int_0^L y(t,x)^2dx,
\end{equation}
and its derivative with respect to the time variable gives
\begin{equation}
\label{derivative-E-saturation}
\dot{E}\leq -a\int_0^L y(t,x)\texttt{sat}(y(t,x))dx .
\end{equation}

Since $\texttt{sat}$ is an odd function, thus, for all $(t,x)\in\mathbb{R}_+\times [0,L]$

\begin{equation}
y(t,x)\texttt{sat}(y(t,x))\geq 0
\end{equation}

%Let us recall a result that comes from \cite[Chapter 1]{tarbouriech2011book_saturating}.
%\begin{nnlemma}
%\label{lemma_dz}
%\cite{tarbouriech2011book_saturating} For all $s\in\mathbb{R}$, the nonlinearity $\varphi$ satisfies \begin{equation*}
%\varphi(s)(\varphi(s)+s)\leq 0.
%\end{equation*}
%\end{nnlemma}
%
%By applying Lemma \ref{lemma_dz} with $s:=y(t,x)$, for all $b>0$, we get that
%\begin{equation}
%\begin{split}
%\dot{E}\leq &-2a\int_0^L y(t,x)^2 dx-2a\int_0^L y(t,x)\varphi(y(t,x))dx\\
%&-2b\int_0^L (\varphi(y(t,x))^2+\varphi(y(t,x))y(t,x))dx
%\end{split}
%\end{equation}
%that means that we have
%\begin{equation}
%\dot{E}\leq -2\int_0^L \begin{bmatrix}
%y(t,x)\\
%\varphi(y(t,x))
%\end{bmatrix}^T\mathcal{A}\begin{bmatrix}
%y(t,x)\\
%\varphi(y(t,x))
%\end{bmatrix}dx
%\end{equation}
%with $\mathcal{A}$ defined by
%\begin{equation}
%\mathcal{A}=\begin{bmatrix}
%a & \frac{a+b}{2}\\
%\frac{a+b}{2} & b
%\end{bmatrix}.
%\end{equation}
%
%First, let us select $b$ such that $b\neq a$. Thus we have $\texttt{det}(\mathcal{A})$ strictly negative. The two eigenvalues of the matrix have different signs. It does not help us to find a bound to the function $\dot{E}$.
%
%Thus we set $a=b$. Hence, the matrix $\mathcal{A}$ is rewritten as follows $\mathcal{A}=\begin{bmatrix}
%a & a\\
%a & a
%\end{bmatrix}$ that is semi-definite positive (because its eigenvalues are $0$ and $2a$).

Therefore we get
\begin{equation}
\dot{E}\leq 0,
\end{equation}
which means that, for all initial conditions in $D(A)$, the solutions of (\ref{KdV-saturated}) are stable. The attractivity has to be inspected too in order to finish the proof of the stability. 

Since we are in an infinite dimensional context, using the LaSalle's Invariance Principle needs us to check whether the trajectories are compact. This precompactness is a corollary of the following lemma (which is very similar to \cite[Lemma 2]{dandreanovel1994mcss}, where a wave equation is considered). 

\begin{nnlemma}
\label{lemma-precompacity}
The canonical embedding from D(A), equipped with the graph norm, into $L^2(0,L)$ is compact.
\end{nnlemma}

\begin{proof}
Before proving this lemma, recall that its statement is equivalent to prove, for each sequence in $D(A)$, which is bounded with the graph norm, that it exists a subsequence that (strongly) converges in $L^2(0,L)$.

Let us recall the definition of the graph norm
\begin{equation}
\begin{split}
\Vert u\Vert^2_{D(A)}:= &\Vert u\Vert^2_{L^2(0,L)}+\Vert Au\Vert^2_{L^2(0,L)}\\
=&\int_0^L \left(|u|^2+|-u^{\prime\prime\prime}-u^{\prime}-a\texttt{sat}(u)|^2\right)dx\\
=&\int_0^L \left(|u|^2+|u^{\prime\prime\prime}+u^{\prime}+a\texttt{sat}(u)|^2\right)dx.
\end{split}
\end{equation}

Since for all $(s,\tilde{s})\in\mathbb{C}^2$, $|s+\tilde{s}|^2\leq 2|s|^2+2|\tilde{s}|^2$, we get the following two inequalities
\begin{equation}
\label{precompactness.1}
\Vert u\Vert^2_{D(A)}\geq \Vert u\Vert^2_{L^2(0,L)}
\end{equation}
and
\begin{equation}
\label{precompactness.2}
\begin{split}
\Vert u\Vert^2_{D(A)}&\geq \min\left(1,\frac{1}{a}\right)\int_0^L |-a\texttt{sat}(u)|^2dx\\
&+\min\left(1,\frac{1}{a}\right)\int_0^L |u^{\prime\prime\prime}+u^\prime+a\texttt{sat}(u)|^2dx\\
&\geq \min\left(\frac{1}{2},\frac{1}{2a}\right)\int_0^L |u^{\prime\prime\prime}+u^\prime|^2 dx.
\end{split}
\end{equation}

Noticing that $\Vert u^{\prime\prime\prime}\Vert^2_{L^2(0,L)}=\Vert u^{\prime\prime\prime}+u^\prime-u^\prime\Vert^2_{L^2(0,L)}$, we have
\begin{equation}
\label{precompactness.3}
\Vert u^{\prime\prime\prime}\Vert^2_{L^2(0,L)} \leq  2\Vert u^{\prime\prime\prime}+u^\prime\Vert^2_{L^2(0,L)}+2\Vert u^\prime\Vert^2_{L^2(0,L)},
\end{equation}
and using that $\Vert u^\prime\Vert^2_{L^2(0,L)}=\Vert u^\prime+u^{\prime\prime\prime}-u^{\prime\prime\prime}+xu-xu\Vert^2_{L^2(0,L)}$, we obtain
\begin{equation*}
\begin{split}
\Vert u^\prime\Vert^2_{L^2(0,L)}\leq & 2\Vert u^\prime+u^{\prime\prime\prime}\Vert^2_{L^2(0,L)}+2\Vert u^{\prime\prime\prime}-xu+xu\Vert^2_{L^2(0,L)}\\
\leq & 2\Vert u^\prime+u^{\prime\prime\prime}\Vert^2_{L^2(0,L)}+4\Vert u^{\prime\prime\prime}-xu\Vert^2_{L^2(0,L)}\\
&+4\Vert xu\Vert^2_{L^2(0,L)}\\
\leq & 2\Vert u^\prime+u^{\prime\prime\prime}\Vert^2_{L^2(0,L)}+4\Vert u^{\prime\prime\prime}\Vert^2_{L^2(0,L)}\\
&-8\int_0^L xu^{\prime\prime\prime}udx+8\Vert xu\Vert^2_{L^2(0,L)}.
\end{split}
\end{equation*} 

From (\ref{lastuce}), we get
\begin{equation*}
\int_0^L xu^{\prime\prime\prime}udx=\frac{3}{2}\Vert u^\prime\Vert^2_{L^2(0,L)}
\end{equation*}
and therefore 
\begin{equation}
\label{precompactness.4}
\begin{split}
\Vert u^\prime\Vert^2_{L^2(0,L)}\leq & 2\Vert u^\prime+u^{\prime\prime\prime}\Vert^2_{L^2(0,L)}+4\Vert u^{\prime\prime\prime}\Vert^2_{L^2(0,L)}\\
&-12\Vert u^{\prime}\Vert^2_{L^2(0,L)}+8\Vert xu\Vert^2_{L^2(0,L)}.
\end{split}
\end{equation}

Thus:

\begin{equation}
\label{precompactness.5}
\begin{split}
13\Vert u^\prime\Vert^2_{L^2(0,L)}\leq & 2\Vert u^\prime+u^{\prime\prime\prime}\Vert^2_{L^2(0,L)}+4\Vert u^{\prime\prime\prime}\Vert^2_{L^2(0,L)}\\
&+8L^2\Vert u\Vert^2_{L^2(0,L)}
\end{split}
\end{equation} 

Plugging inequality (\ref{precompactness.3}) in (\ref{precompactness.5}), we have

\begin{equation*}
\begin{split}
13\Vert u^\prime\Vert^2_{L^2(0,L)}\leq & 2\Vert u^\prime+u^{\prime\prime\prime}\Vert^2_{L^2(0,L)}\\
&+4\left(2\Vert u^{\prime\prime\prime}+u^\prime\Vert^2_{L^2(0,L)}+2\Vert u^\prime\Vert^2_{L^2(0,L)}\right)\\
& +8L^2\Vert u\Vert^2_{L^2(0,L)}\\
\leq & 10\Vert u^\prime+u^{\prime\prime\prime}\Vert^2_{L^2(0,L)}+8\Vert u^\prime\Vert^2_{L^2(0,L)}\\
&+8L^2\Vert u\Vert^2_{L^2(0,L)}.
\end{split}
\end{equation*}
and therefore
\begin{equation}
\Vert u^\prime\Vert^2_{L^2(0,L)}\leq 2\Vert u^\prime+u^{\prime\prime\prime}\Vert^2_{L^2(0,L)}+\frac{8L^2}{5}\Vert u\Vert^2_{L^2(0,L)}
\end{equation}
Considering Equations (\ref{precompactness.1}) and (\ref{precompactness.2}), it leads us to the following inequality, for all $u\in D(A)$
\begin{equation}
\label{boundary-graph}
\Vert u^\prime\Vert^2_{L^2(0,L)}\leq \Delta \Vert u\Vert^2_{D(A)}
\end{equation}
where $\Delta$ is a term which depends on $L$ and $a$.

Thus, if we consider now a sequence $\{u_n\}_{n\in\mathbb{N}}$ in $D(A)$ bounded for the graph norm of $D(A)$, we have from (\ref{boundary-graph}) that this sequence is bounded in $H^1_0(0,L)$. Since the canonical embedding from $H^1_0(0,L)$ to $L^2(0,L)$ is compact, there exists a subsequence still denoted $\{u_n\}_{n\in\mathbb{N}}$ such that $u_n\rightarrow u$ in $L^2(0,L)$. Thus $u$ belongs to $L^2(0,L)$, which concludes the lemma. 
\end{proof}

Now we apply the LaSalle's Invariance Principle. %Using the dissipativity (see Lemma \ref{lemma-dissipativity}), 

Using the fact that $A$ generates a semi-group of contraction, then from \cite[Th\'eor\`eme 3.1, Page 54]{brezis1973operateurs}, we get, for all $t\geq 0$ and for all $y(0,\cdot)\in D(A)$,

\begin{equation}
\Vert y(t,\cdot)\Vert_{L^2(0,L)}\leq \Vert y(0,\cdot)\Vert_{L^2(0,L)}
\end{equation} 

and

\begin{equation}
\Vert Ay(t,\cdot)\Vert_{L^2(0,L)}\leq \Vert Ay(0,\cdot)\Vert_{L^2(0,L)}.
\end{equation}

Therefore, thanks to Lemma \ref{lemma-precompacity}, we see that the trajectory $\lbrace v(t)=S(t)v_0,\: t\geq 0\rbrace$ is precompact in $L^2(0,L)$, then the $\omega$-limit set $w[(y(0,\cdot))]\subset D(A)$, is not empty and invariant to the nonlinear semigroup $S(t)$ (see \cite[Theorem 3.1]{slemrod1989mcss}).

Let us consider a strong solution such that $\dot{E}(t)=0$, for all $t\geq 0$. It follows from (\ref{derivative-E-saturation}) that $y(t,x)=0$ for almost $x$ in $(0,L)$. %and thus the absolute value of $y$ is smaller than the level of saturation map, i.e. $|y|\leq u_0$. Thus, for sufficiently large time, $y$ is a solution to the linear system (\ref{linear-KdV-controlled}) which is asymptotically stable. 
Therefore the convergence property (\ref{convergence-property}) holds along the strong solutions to the nonlinear equation (\ref{KdV-saturated}). 

Using the density of $D(A)$ and the existence of weak solutions, we end the proof by extending the result to any initial condition in $L^2(0,L)$. 

%\begin{nnremark}
%The result still holds for feedbacks $f(t,x)=a.h(y(t,x))$ where $h$ satisfies for all $s\in\mathbb{R}$
%\begin{itemize}
%\item[(i)] $h(-s)=-h(s)$;
%\item[(ii)] $|h(s)|\leq s$;
%\item[(iii)]$\Vert h(s)\Vert^2_{L^2(0,L)}\leq K$, where $K>0$ is a constant;
%\item[(iv)] $h$ is globally Lipschitz.
%\end{itemize}
%\end{nnremark} 

\section{Simulation}\label{simulation}

Let us discretize the PDE (\ref{KdV-saturated}) by means of finite difference method (see e.g. \cite{nm_KdV} for an introduction on the numerical scheme of a generalized Korteweg-de Vries equation). The time and the space steps are chosen such that the stability condition of the numerical scheme is satisfied.

We choose $L=2\pi$, $y_0(x)=100(1-\cos(x))$ for all $x\in [0,2\pi]$ and $a=1$. Let us numerically compute the solution of (\ref{KdV-saturated}). On Figure \ref{figure1}, there is no saturation in the dynamics. On Figure \ref{figure2}, there is a saturation with a level $u_0=1$. On Figure \ref{figure3}, the feedback law is saturated with a level $u_0=3$. Figures \ref{figure4} and \ref{figure5} illustrate the evolution of the Lyapunov function $E$ with respect to the time (without saturation and with a saturation level equals to $3$). 

\begin{figure}[ht]
  \includegraphics[scale=0.6]{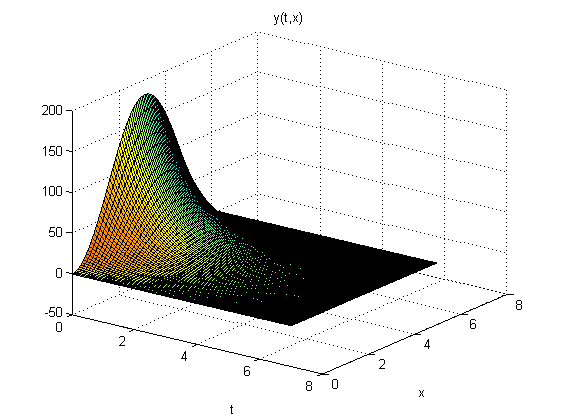}
\caption{Solution y(t,x) with a feedback law without saturation}
\label{figure1}
\end{figure}

\begin{figure}[ht]
  \includegraphics[scale=0.6]{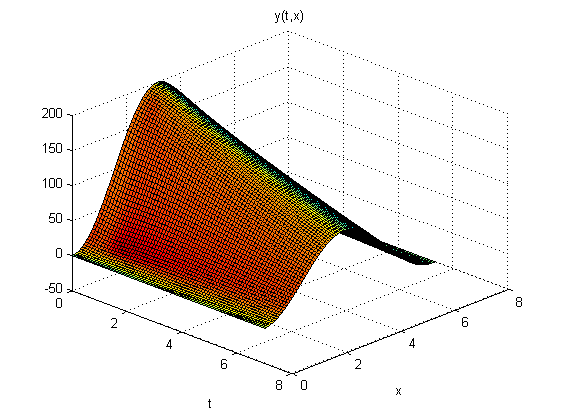}
\caption{Solution y(t,x) with a saturated feedback law and $u_0=1$.}
\label{figure2}
\end{figure}

\begin{figure}[ht]
  \includegraphics[scale=0.6]{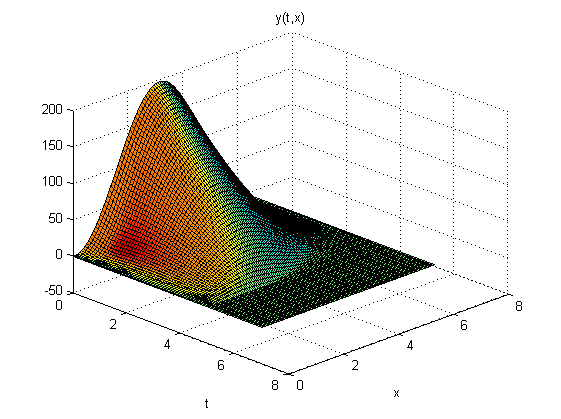}
\caption{Solution y(t,x) with a saturated feedback law and $u_0=3$.}
\label{figure3}
\end{figure}

\begin{figure}
   \begin{minipage}[c]{.46\linewidth}
      \includegraphics[scale=0.33]{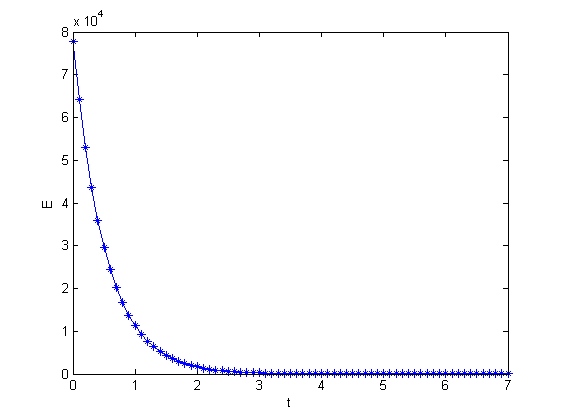}
      \caption{Time evolution of the Lyapunov function $E$ without saturation.}
      \label{figure4}
   \end{minipage} \hfill
   \begin{minipage}[c]{.46\linewidth}
      \includegraphics[scale=0.33]{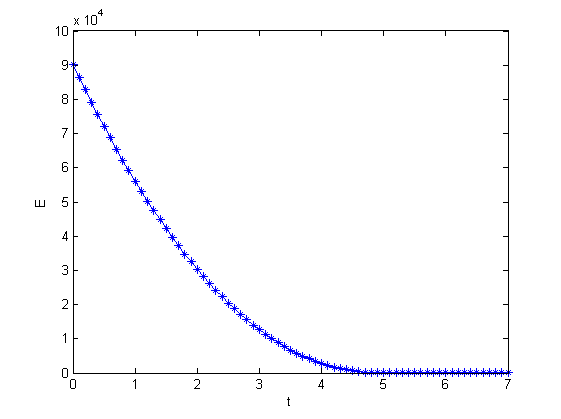}
      \caption{Time evolution of the Lyapunov function $E$ with a saturation $u_0=3$.}
      \label{figure5}
   \end{minipage}
\end{figure}

\section{Conclusion}\label{part5}
In this paper, we have studied the well-posedness and the asymptotic stability of a linear Korteweg-de Vries equation with a saturated distributed control. The well-posedness issue has been tackled by using the nonlinear semigroup theory and we proved the stability by using a sector condition and Lyapunov theory for infinite dimensional system. We illustrate our result on some simulations, which show that the smaller is the saturation level, the slower is the convergence to zero.

To conclude, let us state some questions arising in this context:\\
1. Can we extend our theorems to the nonlinear Korteweg-de Vries equation? \\
2. As mentioned in the introduction, even if the internal control (without any constraints) acts only on a part of the domain, the stability still holds. Is it true with a saturated control? \\
3. Can we recover the exponential stability with the saturated input? \\
4. Can we apply the same method for other partial differential equations? An interesting model could be the one-dimensional Kuramoto-Sivashinky equation.

\bibliographystyle{plain}
\bibliography{Bibliographie}

\end{document}